\pdfoutput=1
\documentclass[a4paper,USenglish,cleveref, autoref, thm-restate,nolineno]{socg-lipics-v2021}
\hideLIPIcs

\title{Linear relations between face numbers of 
levels in arrangements}

\titlerunning{Linear relations between face numbers of levels in arrangements} 

\author{Elizaveta Streltsova}{ISTA, Klosterneuburg, Austria}{estrelts@ist.ac.at
}{}{}

\author{Uli Wagner}{ISTA, Klosterneuburg, Austria \and \url{https://ist.ac.at/en/research/wagner-group/}}{uli@ist.ac.at}{https://orcid.org/0000-0002-1494-0568}{}

\authorrunning{E. Streltsova and U. Wagner} 

\Copyright{Elizaveta Streltsova and Uli Wagner} 



\category{} 

\relatedversion{} 

\newcounter{sideremark}

\newcommand{\arr}{\mathcal{A}}

\DeclareMathOperator{\sgn}{sgn}
\DeclareMathOperator{\conv}{conv}

\DeclareMathOperator{\aff}{aff}
\DeclareMathOperator{\lin}{lin}

\newcommand{\fpoly}{f}

\newcommand{\fspace}{\mathfrak{F}}

\newcommand{\gspace}{\mathfrak{G}}

\newcommand{\fhom}{p}

\renewcommand{\leq}{\leqslant}

\renewcommand{\geq}{\geqslant}

\renewcommand{\phi}{\varphi}

\newcommand{\R}{\mathbf{R}}
\newcommand{\Z}{\mathbf{Z}}

\usepackage{mathtools}

\DeclarePairedDelimiter\floor{\lfloor}{\rfloor}

\newcommand{\cyclic}{V_{\textup{cyclic}}}
\newcommand{\cocyclic}{V_{\textup{cocyclic}}}

\usepackage{appendix}
\usepackage[normalem]{ulem}
\usepackage{hyperref}

\begin{document}

\maketitle
\thispagestyle{empty}

\begin{abstract}
We study linear relations between face numbers of levels in arrangements. Let $V = \{ v_1, \ldots, v_n \} \subset \mathbf{R}^{r}$ be a vector configuration in general position, and let $\mathcal{A}(V)$ be polar dual arrangement of hemispheres in the $d$-dimensional unit sphere $S^d$, where $d=r-1$. For $0\leq s \leq d$ and $0 \leq t \leq n$, 
let $f_{s,t}(V)$ denote the number of faces of \emph{level} $t$ and dimension $d-s$ in the arrangement $\mathcal{A}(V)$ (these correspond to partitions $V=V_-\sqcup V_0 \sqcup V_+$ by linear hyperplanes with $|V_0|=s$ and $|V_-|=t$).
We call the matrix $f(V):=[f_{s,t}(V)]$ the \emph{$f$-matrix} of $V$. 

Completing a long line of research on linear relations between face numbers of levels in arrangements, we determine, for every $n\geq r \geq 1$, the affine space $\fspace_{n,r}$ spanned by the $f$-matrices of configurations of $n$ vectors in general position in $\R^r$; moreover, we determine the subspace $\fspace^0_{n,r} \subset \fspace_{n,r}$ spanned by all \emph{pointed} vector configurations (i.e., such that $V$ is contained in some open linear halfspace), which correspond to point sets in $\R^d$. This generalizes the classical fact that the \emph{Dehn--Sommerville relations} generate all linear relations between the face numbers of simple polytopes 
and answers a question posed by Andrzejak and Welzl in 2003.

The key notion for the statements and the proofs of our results is the \emph{$g$-matrix} of a vector configuration, which determines the $f$-matrix and generalizes the classical \emph{$g$-vector} of a polytope.

By Gale duality, we also obtain analogous results for partitions of vector configurations by sign patterns of nontrivial linear dependencies, and for \emph{Radon partitions} of point sets in $\mathbf{R}^d$.
\end{abstract}

\section{Introduction}
\label{s:intro}
Levels in arrangements (and the dual notion of $k$-sets) play a fundamental role in discrete and computational geometry 
and 
are a natural generalization of convex polytopes (which correspond to the $0$-level); we refer to  \cite{Edelsbrunner:1987aa,Matousek:2002aa,Wagner:2008aa} for more background.

It is a classical result in polytope theory that the \emph{Euler--Poincaré relation} is the only linear relation between the face numbers of arbitrary $d$-dimensional convex polytopes, and that the \emph{Dehn--Sommerville relations} 
(which we will review below) 
generate all linear relations between the face numbers of \emph{simple} (or, dually, \emph{simplicial}) polytopes \cite[Chs.~8--9]{Grunbaum:2003aa}. 

Our focus is on understanding, more generally, the linear relations between face numbers of levels in simple arrangements. These have been studied extensively \cite{Mulmuley:1993aa,Gulliksen:1997aa,Linhart:2000aa,Andrzejak:2003aa,Biswas:2024aa}; in particular, 
generalizations of the Dehn--Sommerville relations for face numbers of higher levels were proved, in different forms, by Mulmuley \cite{Mulmuley:1993aa} and by Linhart, Yang, and Philipp \cite{Linhart:2000aa} (see also \cite{Biswas:2024aa}). Here, we determine, in a precise sense, all linear relations, which answers a question posed, e.g., by Andrzejak and Welzl~\cite{Andrzejak:2003aa} in 2003. To state our results formally, it will be convenient to work in the setting of \emph{spherical arrangements} in $S^d$ (which can be seen as a ``compactification'' of arrangements of affine hyperplanes or halfspaces in $\R^d$; this setting is more symmetric and avoids technical issues related to \emph{unbounded} faces). 

\subsection{Levels in Arrangements, Dissection Patterns, and Polytopes}
\label{sec:arrangments}
Throughout this paper, let $r=d+1 \geq 1$, and let $S^d$ be the unit sphere in $\R^r$. We denote the standard inner product in $\R^r$ by $\langle\cdot ,\cdot \rangle$, and write $\sgn(x)\in \{-1,0,+1\}$ for the \emph{sign} of a real number 
$x\in \R$. 
For a sign vector $F\in \{-1,0,+1\}^n$, let $F_+$, $F_0$, and $F_-$ denote the subsets of coordinates $i\in [n]$ such that $F_i=+1$, $F_i=0$, and $F_i=-1$, respectively.

Let $V = \{ v_1, \ldots, v_n \} \subset \R^r$ be a set of $n\geq r$ vectors; fixing the labeling of the vectors, we will also view $V$ as an $(r\times n)$-matrix $V=[v_1|\dots|v_n] \in \R^{r\times n}$ with column vectors $v_i$. Unless stated otherwise, we assume that $V$ is in general position, i.e., that any $r$ of the vectors are linearly independent. We refer to $V$ as a \emph{vector configration} of $\emph{rank}$ $r$. 

We consider the partitions of $V$ by linear hyperplanes, or equivalently, the faces of the dual arrangement. Formally, every vector $v_i \in V$ defines a great $(d-1)$-sphere
$
H_i = \{ x \in S^d \mid \langle v_i, x \rangle = 0 \}
$
in $S^d$ and two open hemispheres
$$
H_i^+ = \{ x \in S^d \mid \langle v_i, x \rangle >0 \}, \qquad H_i^- = \{ x \in S^d \mid \langle v_i, x \rangle < 0 \}.
$$
The resulting \emph{arrangement} $\arr(V)=\{H_1^+,\dots,H_n^+\}$ of hemispheres in $S^d$ determines a decomposition of $S^d$ into \emph{faces} of dimensions $0$ through $d$, where two points $u,u'\in S^d$ lie in the relative interior of the same face iff $\sgn(\langle v_i,u \rangle)=\sgn(\langle v_i,u'\rangle)$ for $1\leq i\leq n$. Let $\mathcal{F}(V)$ be the set of all sign vectors $(\sgn(\langle v_1,u\rangle),\dots,\sgn(\langle v_n,u\rangle) \in \{-1,0,+1\}^n$, where $u$ ranges over all non-zero vectors (equivalently, unit vectors) in $\R^r$. We can identify each face of $\arr(V)$ with its \emph{signature} $F\in \mathcal{F}(V)$; by general position, the face with signature $F$ has dimension $d-|F_0|$ (there are no faces with $|F_0|>d$, i.e., the arrangement is \emph{simple}). Moreover, we call $|F_-|$ the \emph{level} of the face. Equivalently, the elements of $\mathcal{F}(V)$ correspond bijectively to the partitions of $V$ by oriented linear hyperplanes, and we will also call them the \emph{dissection patterns} of $V$. In what follows, we will pass freely back and forth between a vector configuration $V$ and the corresponding arrangement $\mathcal{A}(V)$ and refer to this correspondence as \emph{polar duality} (to distinguish it from \emph{Gale duality}, see below).

\begin{definition}[$f$-matrix and $f$-polynomial]
\label{def:f-poly}
For integers $s$ and $t$, let%
\footnote{By general position, $f_{s,t}(V)=0$ unless $0\leq s\leq d$ and $0\leq t\leq n-s$, but it will occasionally be convenient to allow an unrestricted range of indices.}
$$
f_{s,t} := f_{s,t}(V):=|\{F\in \mathcal{F}(V)\mid |F_0|=s,|F_-|=t\}.
$$
Thus, $f_{s,t}(V)$ counts the $(d-s)$-dimensional faces of level $t$ in $\arr(V)$. 

Together, these numbers form the \emph{$f$-matrix} $f(V)=[f_{s,t}(V)]$. Equivalently, we can encode this data into the bivariate \emph{$f$-polynomial} $\fpoly_V(x,y) \in  \Z[x,y]$ defined by 
$$
f_{V}(x,y) := \sum_{s,t} f_{s,t}(V)\, x^s y^t =\sum_{F\in \mathcal{F}(V)} x^{|F_0|} y^{|F_-|}.
$$
\end{definition}
\begin{observation} 
\label{obs:f-antipodal-symm} By antipodal symmetry $F\leftrightarrow -F$ of $\mathcal{F}(V)$, 
\begin{equation}
\label{eq:f-symm}
f_{s,t}(V)=f_{s,n-s-t}(V) \quad \textrm{for all $s$ and $t$; equivalently,} \textstyle \quad f_V(x,y)= y^n f_V(\frac{x}{y},\frac{1}{y})
\end{equation}
\end{observation}
It is well-known \cite[Sec.~18.1]{Grunbaum:2003aa} that the total number of faces of a given dimension $d-s$ (of any level) in a simple arrangement in $S^d$ depends only on $n$, $d$, and $s$; more specifically:
\begin{lemma}
Let $\arr(V)$ be a simple arrangement of $n$ hemispheres in $S^d$. Then, for $0\leq s\leq d$,
the total number of $(d-s)$-dimensional faces (of any level) in $\arr(V)$ equals
\begin{equation}
\label{eq:total-number-faces}
\sum_t f_{s,t}(V) = 2\binom{n}{s}\sum_{i=0}^{d-s} \binom{n-s-1}{i} = \sum_{i=0}^d (1+(-1)^i)\binom{n}{d-i} \binom{d-i}{s}
\end{equation}
In terms of the $f$-polynomial, this can be expressed very compactly as
\begin{equation}
\label{eq:f-tutte}
\textstyle f_V(x,1) = \sum_{i=0}^d \binom{n}{i}\left(1+(-1)^{d-i}\right)(1+x)^i = 2\left(\binom{n}{d}(x+1)^d + \binom{n}{d-2}(x+1)^{d-2} + \dots\right)
\end{equation}
\end{lemma}

We call a vector configuration $V$ \emph{pointed} if it is contained in an open linear halfspace $\{x\in \R^r\colon \langle u,x\rangle >0\}$, for some $u\in S^d$, or equivalently, if 
$\bigcap_{i=1}^n H^+_i\neq \emptyset$. 
The closure of this intersection is then a simple (spherical) polytope $P$, the $0$-level of $\arr(V)$.
By radial projection onto the tangent hyperplane $\{x\in \R^r\colon \langle u,x\rangle = 1\}$, every pointed configuration $V\subset \R^r$ corresponds to a point set $S\subset \R^d$, see \cite[Sec.~5.6]{Matousek:2002aa}. 
The convex hull $P^\circ=\conv(S)$ is a simplicial polytope (the polar dual of $P$), and the elements of $\mathcal{F}(V)$ correspond to the partitions of $S$ by oriented affine hyperplanes. 

Linhart, Yang, and Philipp \cite{Linhart:2000aa} proved the following result, which generalizes the classical \emph{Dehn--Sommerville relations} for simple polytopes:
\begin{theorem}[Dehn--Sommerville Relations for Levels in Simple Arrangements]
\label{thm:DS-levels}
Let $V \in \R^{r\times n}$ be a vector configuration in general position. Then
\begin{equation}
\label{eq:DS-f-poly}
f_V(x,y)=(-1)^d f_V(-(x+y+1),y)
\end{equation}
Equivalently (by comparing coefficients), for $0\leq s\leq d$ and $0\leq t\leq n$,
\begin{equation}
\label{eq:DS-f-numbers}
f_{s,t}(V)=\sum_{j} \sum_{\ell} (-1)^{d-j} \binom{j}{s} \binom{j-s}{t-\ell} f_{j,\ell}(V) 
\end{equation}
\end{theorem}
\begin{remark}
The Dehn--Sommerville relations for polytopes correspond to the identity $f_V(x,0)=(-1)^d f_V(-(x+1),0)$. The coefficients on the right-hand side of \eqref{eq:DS-f-numbers} are zero unless $\ell \leq t$ (and $j\geq s$). This yields, for every $k$, a linear system of equations among the numbers $f_{s,t}$, $0\leq s\leq d$ and $t\leq k$, of face numbers of the  \emph{$(\leq k)$-sublevel} of the arrangement $\arr(V)$. An equivalent system of equations (expressed in terms of an \emph{$h$-matrix} that generalizes the $h$-vector of a simple polytope) was proved earlier by Mulmuley~\cite{Mulmuley:1993aa}, under the additional assumption the $(\leq k)$-sublevel is contained in an open hemisphere. 
Related relations have been rediscovered several times (e.g., in the recent work of Biswas et al.~\cite{Biswas:2024aa}).
\end{remark}

The central notion of this paper is the \emph{$g$-matrix} $g(V\to W)$ of a pair $V,W \in \R^{r\times n}$ of vector configurations, which generalizes the \emph{$g$-vector} of a simple polytope and encodes the difference $f(W)-f(V)$ between the $f$-matrices. The $g$-matrix was introduced by Streltsova and Wagner~\cite{SW:Arcs}, who showed that the first quadrant of the $g$-matrix (the \emph{small} $g$-matrix) of every vector configuration in $\R^3$ is non-negative and used this to establish, in the special case $d=n-4$, a conjecture of Eckhoff~\cite{Eckhoff:1993aa}, Linhart~\cite{Linhart:1994aa}, and Welzl~\cite{Welzl:2001aa} generalizing the Upper Bound Theorem for polytopes to sublevels of arrangements in $S^d$.

The geometric definition of the $g$-matrix is given in Sec.~\ref{sec:g-matrix}, based on how the $f$-matrix changes by \emph{mutations} in the course of generic \emph{continuous motion} (an idea with a long history, see, e.g., \cite{Andrzejak:1998aa}). The $g$-matrix is characterized by the following properties:
\begin{theorem} 
\label{thm:f-g}
Let $V,W \in \R^{r\times n}$ be a pair of vector configurations in general position. 

The \emph{$g$-matrix} $g(V\to W)$ of the pair is an $(r+1)\times (n-r+1)$-matrix  with integer entries $g_{j,k}:=g_{j,k}(V\to W)$, $0\leq j\leq r$, $0\leq k \leq n-r$, which has the following properties:  
\begin{enumerate}
\item For $0\leq j\leq r$ and $0\leq k\leq n-r$, the $g$-matrix satifies the skew-symmetries
\begin{equation}
\label{eq:g-skew}
g_{j,k}=-g_{r-j,k}=-g_{j,n-r-k}=g_{r-j,n-r-k}
\end{equation}
Thus, the $g$-matrix is determined by the submatrix $[g_{j,k}\colon 0\leq j \leq \floor{\frac{r-1}{2}}, 0\leq k \leq \floor{\frac{n-r-1}{2}}]$, which we call the \emph{small $g$-matrix}. 
Equivalently, the \emph{$g$-polynomial} $g(x,y):=g_{V\to W}(x,y):=\sum_{j,k} g_{j,k} x^j y^k \in \Z[x,y]$ satisfies
\begin{equation}
\label{eq:g-poly-skew}
\textstyle g(x,y)=-x^rg(\frac{1}{x},y)=-y^{n-r}g(x,\frac{1}{y})=x^ry^{n-r}g(\frac{1}{x},\frac{1}{y})
\end{equation}
\item The $g$-polynomial determines the difference $f_W(x,y)-f_V(x,y)$ of $f$-polynomials by
\begin{equation}
\label{eq:f-poly-g-poly}
f_W(x,y)-f_V(x,y) = (1+x)^r g (\textstyle \frac{x+y}{1+x},y)=\sum_{j=0}^r \sum_{k=0}^{n-r} g_{j,k}\cdot (x+y)^j (1+x)^{r-j} y^k
\end{equation}
Equivalently (by comparing coefficients), for $0\leq s\leq d$ and $0\leq t\leq n$,
\begin{equation}
\label{eq:f-g-coefficients}
f_{s,t}(W)-f_{s,t}(V) = \sum_{j,k} \binom{j}{t-k}\binom{r-j}{s-j+t-k} g_{j,k}(V\to W),
\end{equation}
\item $g(W\to V)=-g(V\to W)$.
\end{enumerate}
\end{theorem}
\begin{remark} 
\label{rem:g-f+DS-g}
The system of equations \eqref{eq:f-g-coefficients} yields a linear transformation $T=T_{n,r}$ through which 
the $g$-matrix $g=g(V\to W)$ of the pair determines the difference $\Delta f=f(W)-f(V)$ of $f$-matrices by $\Delta f=T(g)$. As we will show below (Lemma~\ref{lem:g-f}), in the presence of the skew-symmetries \eqref{eq:g-skew}, the transformation $T$ is injective, i.e., $g(V\to W)$ is uniquely determined by $\Delta f=f(W)-f(V)$. Thus, Theorem~\ref{thm:f-g} could be taken as a formal definition of the $g$-matrix. 
\end{remark}
\begin{remark}
The skew-symmetry $g(x,y)=-x^r g(\frac{1}{x},y)$ reflects the Dehn--Sommerville relation \eqref{eq:DS-f-poly}, and 
the symmetry $g(x,y)=x^ry^{n-r}g(\frac{1}{x},\frac{1}{y})$ reflects the antipodal symmetry \eqref{eq:f-symm}.
\end{remark}

We are now ready to state our main results.

\begin{theorem}
\label{thm:f-space-g-space} 
Let $\mathcal{V}_{n,r}$ be the set of vector configurations $V\in \R^{r\times n}$ in general position. Let 
\[
\fspace_{n,r}:=\aff\{f(V)\colon V\in \mathcal{V}_{n,r}\}, \qquad \gspace_{n,r}:=\lin\{g(V \to W) \colon V,W \in \mathcal{V}_{n,r}\} 
\]
be the affine space spanned by all $f$-matrices and the linear space spanned by all $g$-matrices of pairs, respectively.
Then $\dim \fspace_{n,r}=\dim \gspace_{n,r}=\floor{\frac{r+1}{2}}\floor{\frac{n-r+1}{2}}$; more precisely, 
\begin{equation}
\label{eq:g-space}
\gspace_{n,r}=\left\{ g \in \R^{(r+1)\times (n-r+1)} \colon \begin{array}{c}g_{j,k}=-g_{r-j,k}=-g_{j,n-r-k}=g_{r-j,n-r-k}\\ \textrm{for } 0\leq j\leq r, 0\leq k\leq n-r\end{array}
\right\}
\end{equation}
is the space of all \emph{real} $(r+1)\times (n-r+1)$-matrices satisfying the skew-symmetries \eqref{eq:g-skew}, and 
\[
\fspace_{n,r}=f(V_0)+T(\gspace_{n,r})
\]
for any fixed $V_0 \in \mathcal{V}_{n,r}$, where $T=T_{n,r}$ is the injective linear transformation given by \eqref{eq:f-g-coefficients}.
\end{theorem}

\begin{theorem}
\label{thm:f-space-g-space-pointed}
Let $\mathcal{V}^0_{n,r} \subset \mathcal{V}_{n,r}$ be the subset of pointed configurations (corresponding to point sets in $\R^d$, $d=r-1$), and let  
\[
\fspace^0_{n,r}:=\aff\{f(V)\colon V\in \mathcal{V}^0_{n,r}\}, \qquad \gspace^0_{n,r}:=\lin\{g(V \to W) \colon V,W \in \mathcal{V}^0_{n,r}\}
\]
be the corresponding subspaces of $\fspace_{n,r}$ and $\gspace_{r,n}$. Then $\dim \fspace^0_{n,r}=\dim \gspace^0_{n,r}=\floor{\frac{r-1}{2}}\cdot \floor{\frac{n-r+1}{2}}$. More precisely, 
\[
\gspace^0_{n,r}=\{g\in \gspace_{n,r}\colon g_{0,k}=0, 0\leq k \leq n-r\},\quad \textrm{and}\quad \fspace^0_{n,r}=f(V_0)+T(\gspace^0_{n,r})
\]
for any $V_0 \in \mathcal{V}^0_{n,r}$.
\end{theorem}

\begin{remark}
As a specific base configuration $V_0$ in both theorems, one can take the \emph{cyclic} vector configuration $\cyclic(n,r)$ (see Example~\ref{ex:cyclic-cocyclic}), whose $f$-matrix is known explicitly \cite{Andrzejak:2003aa}. 
\end{remark}

\subsection{Dependency Patterns and Radon Partitions}
Let $V = \{ v_1, \ldots, v_n \} \subset \R^r$ be a vector configuration in general position, and let $\mathcal{F}^*(V)$ be the set of all sign vectors $(\sgn(\lambda_1),\dots,\sgn(\lambda_n)) \in \{-1,0,+1\}^n$ given by non-trivial linear dependencies $\sum_{i=1}^n \lambda_i v_i=0$ (with coefficients $\lambda_i\in \R$, not all of them are zero). We call $\mathcal{F}^*(V)$ the \emph{dependency patterns} of $V$. If $V$ is a pointed configuration corresponding to a point set $S\subset \R^d$, $d=r-1$, then the elements of $\mathcal{F}^*(V)$ encode the sign patterns of affine dependencies of $S$, hence they correspond bijectively to (ordered) \emph{Radon partitions} $S=S_-\sqcup S_0 \sqcup S_+$, $\conv(S_+) \cap \conv(S_-)\neq \emptyset$. 

Both $\mathcal{F}(V)$ and $\mathcal{F}^*(V)$ are invariant under invertible linear transformations of $\R^r$ and under positive rescaling (multiplying each vector $v_i$ by some positive scalar $\alpha_i>0$).

\begin{definition}[$f^*$-matrix and $f^*$-polynomial] 
\label{def:fstar-poly}
For integers $s$ and $t$, define\footnote{Note that $f^*_{s,t}(V)=0$ unless $r+1\leq s\leq n$ and $0\leq t\leq s$.%
}
\[
f^*_{s,t}(V):=|\{F\in \mathcal{F}^*(V)\mid |F_-|=t, |F_+|=s-t\}
\]
Together, these numbers form the \emph{$f^*$-matrix} $f^*(V)=[f^*_{s,t}(V)]$. Equivalently, we can encode this data into the bivariate  \emph{$f^*$-polynomial} $f^*_V(x,y) \in \Z[x,y]$ defined by 
\[
f^*_{V}(x,y) := \sum_{F\in \mathcal{F}^*(V)} x^{|F_0|} y^{|F_-|} = \sum_{s,t} f^*_{s,t}(V)\, x^{n-s} y^t
\]
\end{definition}

Given a vector configuration $V \in \R^{r\times n}$ of rank $r$, there is a \emph{Gale dual} vector configuration $V^*  \in \R^{(n-r)\times n}$, whose definition and properties we will review in Section~\ref{sec:Gale}. The configurations $V$ and $V^*$ determine each other (up to invertible linear transformations on $\R^r$ and $\R^{n-r}$, respectively), and they satisfy 
$(V^*)^*=V$ and $\mathcal{F}^*(V)=\mathcal{F}(V^*)$ (hence also $\mathcal{F}(V)=\mathcal{F}^*(V^*)$). It follows that
$f^*_{s,t}(V)=f_{n-s,t}(V^*)$ for all $s,t$, and 
\[
f^*_V(x,y)=f_{V^*}(x,y)
\]
Therefore, by Gale duality, Theorem~\ref{thm:f-space-g-space} immediately gives a complete description of the affine space 
$\fspace^*_{n,r}$ spanned by the $f^*$-matrices of vector configurations $V\in \R^{r\times n}$. However, Gale duals of pointed vector configurations are not pointed, hence a bit more is needed to get a description of the subspace 
$(\fspace^*_{n,r})^0$ spanned by the $f^*$-matrices of pointed vector configurations in $\R^r$ (which count the number of Radon partitions of given types for the corresponding point sets in $\R^d$). We will show the following analogue of Thm~\ref{thm:f-g}, Eq.~\eqref{eq:f-poly-g-poly}.

\begin{theorem}
\label{thm:f*-g*}
Let $V,W$ be configurations of $n$ vectors in $\R^r$. Then
\begin{equation}
\label{eq:f*-g*}
f^*_W(x,y) - f^*_{V}(x,y)= \sum_{j,k} \underbrace{g_{j,k}(W \to V)}_{=-g_{j,k}(V \to W)} (x+y)^k (x+1)^{n-r-k} y^j
\end{equation}
\end{theorem}
As before, \eqref{eq:f*-g*} implies that the difference $f^*=f^*(W)-f^*(V)$ of $f^*$-matrices is the image of the $g$-matrix $g=g(V\to W)$ under an injective linear transformation $S=S_{n,r}$, $f^*=S(g)$. Thus, we get the following:

\begin{theorem}
\label{thm:fstar-space-gstar-space-pointed}
Let $\mathcal{V}^0_{n,r} \subset \mathcal{V}_{n,r}$ be the subset of pointed configurations, and let  
\[
(\fspace^*_{n,r})^0:=\aff\{f^*(V)\colon V\in \mathcal{V}^0_{n,r}\}
\]
be the corresponding subspace of $\fspace^*_{n,r}$. Then $\dim (\fspace^*_{n,r})^0=\floor{\frac{r-1}{2}}\cdot \floor{\frac{n-r+1}{2}}$ and
\[ 
(\fspace^*_{n,r})^0=f^*(V_0)+S(\gspace^0_{n,r})
\]
for any $V_0 \in \mathcal{V}^0_{n,r}$, where $S=S_{n,r}$ is the injective linear transformation given by \eqref{eq:f*-g*}.
\end{theorem}
\begin{remark}\label{rem:OM}
The sets $\mathcal{F}(V)$ and $\mathcal{F}^*(V)$ determine each other (see Lemma~\ref{lem:Farkas}), and analogously for the $f$-matrix and the $f^*$-matrix (Theorem~\ref{thm:f-fstar}). We say that two vector configurations $V,W \in \R^{r\times n}$ have the same \emph{combinatorial type} if (up to a permutation of the vectors) $\mathcal{F}(V)=\mathcal{F}(W)$ (equivalently,  $\mathcal{F}^*(V)=\mathcal{F}^*(W)$). We call $V$ and $W$ \emph{weakly equivalent} if they have identical $f$-matrices (equivalently, identical $f^*$-matrices).

For readers familiar with \emph{oriented matroids} (see \cite[Ch.~6]{Ziegler:1995aa} or \cite{Bjorner:1999aa}), $\mathcal{F}^*(V)$ and $\mathcal{F}(V)$ are precisely the sets of \emph{vectors} and \emph{covectors}, respectively, of the oriented matroid realized by $V$. However, speaking of ``(co)vectors of a vector configuration'' seems potentially confusing, and we hope that the terminology of dissection and dependency patterns is more descriptive. The Dehn--Sommerville relations hold for (uniform, not necessarily realizable) oriented matroids. We believe that Theorems~\ref{thm:f-space-g-space}, \ref{thm:f-space-g-space-pointed}, and \ref{thm:fstar-space-gstar-space-pointed} can also be generalized to that setting.
We plan to treat this in detail in a future paper.
\end{remark}
The remainder of the paper is structured as follows. We present some general background, in particular regarding Gale duality and neighborly and coneighborly configurations, in Section~\ref{sec:Gale}. In Section~\ref{sec:g-matrix}, we give the geometric definition of the $g$-matrix through continuous motion and prove Theorems~\ref{thm:f-g} and \ref{thm:f*-g*}. 
In Section~\ref{sec:g-space}, we then prove Theorems~\ref{thm:f-space-g-space}, \ref{thm:f-space-g-space-pointed}, and \ref{thm:fstar-space-gstar-space-pointed}.

\section{Gale Duality and Neighborly and Coneighborly Configurations}
\label{sec:Gale}

\begin{definition}[Gale Duality] 
\label{def:Gale}
Two vector configurations $V \in \R^{r\times n}$ and $W\in \R^{(n-r)\times n}$ are called \emph{Gale duals} of one another if the rows of $V$ and the rows of $W$ span subspaces of $\R^n$ that are orthogonal complements of one another. Since we always assume that $V$ and $W$ are in general position and of full rank, this is equivalent to the condition $VW^{\top} = 0$.
\end{definition}

It is well-known  that 
Gale dual configurations determine each other up to linear isomorphisms of their ambient spaces $\R^r$ and $\R^{n-r}$, respectively \cite[Sec.~5.6]{Matousek:2002aa}. Thus, we will speak of the \emph{Gale dual} of $V$, which we denote by $V^*$. Obviously, $(V^*)^*=V$.
We have 
\[\mathcal{F}^{*}(V)=\mathcal{F}(V^*),\qquad \textrm{hence} \qquad f^*_{s,t}(V)=f_{n-s,t}(V^*)\quad \textrm{for all }s,t.\]

Let $V=\{v_1,\dots,v_n\}\subseteq \R^r$ be a configuration of $n$ vectors in general position. We call a subset $W\subseteq V$ \emph{extremal} if there exists a linear hyperplane $H$ that contains all vectors in $W$ and such that one of the two open halfspaces bounded by $H$ contains all the remaining vectors in $V\setminus W$, i.e., $G\in \mathcal{F}(V)$, where $G$ is the sign vector with $G_0=\{i\colon v_i \in W\}$ and $G_+=\{i\colon v_i \in V\setminus W\}$. In particular, $V$ is pointed iff the empty subset $\emptyset$ is extremal.

For sign vectors $F,G \in \{-1,0,+\}^n$, write $F\leq G$ if $F_+ \subseteq G_+$ and $F_- \subseteq G_-$. As mentioned above, the sets $\mathcal{F}(V)$ and $\mathcal{F}^*(V)$ determine each other by the following lemma. 
\begin{lemma}
\label{lem:Farkas}
Let $F\in \{-1,0,+1\}^n$. Then $F\not\in \mathcal{F}^*(V)$ iff $F\leq G$ for some $G\in \mathcal{F}(V)$, and $F\not\in \mathcal{F}(V)$ iff $F\leq G$ for some $G\in \mathcal{F}^*(V)$
\end{lemma}
\begin{proof}
It suffices to prove the first equivalence (the second follows by Gale duality). Set $Y:=\{v_i \colon i\in F_+\} \sqcup \{-v_i \colon i\in F_-\}$. We have $F\in \mathcal{F}^*(V)$ iff the origin $0\in \R^r$ can be written as a linear combination 
$0=\sum_{y \in Y} \lambda_y y$ with all coefficients $\lambda_y>0$, which is the case iff $0$ lies in the convex hull $\conv(Y)$. Thus, $F\not\in \mathcal{F}^*(V)$ iff $0\not\in \conv(Y)$, which means that $Y$ is contained in an open linear halfspace $\{x\in \R^r\colon \langle u,x\rangle >0\}$, for some non-zero $u\in \R^r$; equivalently, the sign vector $G\in \mathcal{F}(V)$ given by $G_i =\sgn(\langle u,v_i\rangle)$ satisfies $F\leq G$.
\end{proof}
\begin{corollary}
$V$ is pointed (equivalently, $f_{0,0}(V)=1$) if and only if $f^*_{n,0}(V)=0$.
\end{corollary}

\begin{lemma} 
\label{lem:char-extremal}
$W \subset V$ is extremal iff there is no $F\in \mathcal{F}^*(V)$ with $F_- \subseteq \{i\colon v_i \in W\}$.
\end{lemma}
\begin{proof}
Suppose that $W$ is extremal. Let $u\in \R^r$ be a non-zero vector witnessing this, i.e., $\langle u, v\rangle =0$ for $v \in W$ and $\langle u, v\rangle >0$ for $v\in V\setminus W$. By general position, $|W|\leq r-1$ and $W$ is linearly independent. Thus, if $\sum_{v\in V} \lambda_v v=0$ is a non-trivial linear dependence with $\{v \colon \lambda_v<0\}\subseteq W$, then there must be some $v \in V\setminus W$  with $\lambda_v > 0$. But then $0=\langle u,0\rangle = \sum_{v\in V} \lambda_v \langle u,v\rangle =\sum_{v\in V\setminus W} \lambda_v \langle u,v\rangle >0$, a contradiction.

Conversely, if $W$ is not extremal then, by Lemma~\ref{lem:Farkas}, there exists $F \in \mathcal{F}^*(V)$ with $\{i \colon v_i\in V\setminus W\} \subseteq F_+$, i.e., $F_- \subseteq \{i\colon v_i \in W\}$.
\end{proof}

\begin{definition}[{Neighborly and Coneighborly Configurations}] A vector configuration $V\in \R^{r \times n}$ is \emph{coneighborly} if $f_{s,t}(V) = 0$ for $t \leq \lfloor \frac{n-r-1}{2} \rfloor$, i.e., if every open linear halfspace contains at least $\lfloor \frac{n-r+1}{2} \rfloor$ vectors of $V$. 

We say that  $V \subset \R^r$ is \emph{neighborly} if every subset $W \subseteq V$ of size $|W| \leq \lfloor \frac{r-1}{2} \rfloor$ is \emph{extremal}.
\end{definition}

As a direct corollary of Lemma~\ref{lem:char-extremal}, we get:
\begin{corollary} 
\label{cor:neighborly-alternative}
A vector configuration  $V$ is neighborly iff $f^*_{s,t}(V)=0$ for $t \leq \lfloor \frac{r-1}{2} \rfloor$. Thus, $V$ is neighborly iff its Gale dual $V^*$ is coneighborly.
\end{corollary}

Every neighborly vector configuration $V \subset \R^r$ is pointed, hence corresponds to a point set $S\subset \R^d$, $d=r-1$, and $V$ being neighborly means that every subset of $S$ of size at most $\floor{\frac{r-1}{2}}=\floor{\frac{d}{2}}$ forms a face of the simplicial $d$-polytope $P=\conv(S)$ (which is a \emph{neighborly polytope}). We note that for $r=1,2$ ($d=0,1$) neighborliness is the same as being pointed, and for $r=3,4$ ($d=2,3$) $V$ is neighborly iff the point set $S$ is in convex position.

\begin{example}[{Cyclic and Cocyclic Configurations}]
\label{ex:cyclic-cocyclic}
Let $t_1<t_2<\dots<t_n$ be real numbers and define $v_i:= (1,t_i,t_i^2,\dots,t_i^{r-1}) \in \R^r$. We call $\cyclic(n,r):=\{v_1,\dots v_n\}$ and $\cocyclic(n,r):=\{(-1)^i v_i\colon 1\leq i\leq n\}$ the \emph{cyclic} and \emph{cocyclic} configurations of $n$ vectors in $\R^r$, respectively. Cyclic configurations are neighborly and cocyclic configurations are  
coneighborly \cite[Cor. 0.8]{Ziegler:1995aa}. (Moreover, the combinatorial types of these configurations are independent of the choice of the parameters $t_i$.)
\end{example}

\begin{theorem} 
\label{thm:f-fstar}
Let $V \in \R^{r\times n}$ be a vector configuration in general position. Then the polynomials $f_V(x,y)$ and $f^*_V(x,y)$ determine each other. More precisely,
\begin{equation}
\label{eq:fstar-f}
\hfill
f^*_V(x,y)=(x+y+1)^n - (-1)^r x^n -(x+1)^n f_V(\textstyle -\frac{x}{x+1},\frac{x+y}{x+1})
\hfill
\end{equation}
and
\begin{equation}
\label{eq:f-fstar}
\hfill
f_V(x,y)=(x+y+1)^n - (-1)^{n-r} x^n -(x+1)^n f^*_V(\textstyle -\frac{x}{x+1},\frac{x+y}{x+1})
\hfill
\end{equation}
\end{theorem}

\begin{proof} By Gale duality, it suffices to prove \eqref{eq:fstar-f} (since $f^*_V(x,y)=f_{V^*}(x,y)$).

For this proof, it will be convenient to work with a homogeneous version of both polynomials obtained by associating with each sign vector $F\in \{-1,0,+1\}^n$ the monomial $x^{|F_+|}y^{|F_-|}z^{|F_0|}$. Formally, 
define
\[
\fhom(x,y,z):=\sum_{F\in \mathcal{F}(V)} x^{|F_+|}y^{|F_-|}z^{|F_0|}, \quad \textrm{and}\quad 
\fhom^*(x,y,z):=\sum_{F\in \mathcal{F}^*(V)} x^{|F_+|}y^{|F_-|}z^{|F_0|} 
\]
Using the abridged notation $f=f_V$ and $f^*=f^*_V$, 
we get 
\begin{equation}
\label{eq:inhom-hom}
\begin{split}
\textstyle \fhom(x,y,z) = x^n f(\frac{z}{x}, \frac{y}{x}),\quad &f(z,y)=\fhom(1,y,z), \\
\textstyle \fhom^*(x,y,z) = x^n f^*(\frac{z}{x},\frac{y}{x}),\quad &f^*(z,y)=\fhom^*(1,y,z)
\end{split}
\end{equation}
We start with the observation that $\sum_{F\in \{-1,0,+1\}^n} x^{|F_+|} y^{|F_-|} z^{|F_0|}= (x+y+z)^n$. Therefore, 
\begin{equation}
\label{eq:fstarcomplement}
\fhom^*(x,y,z) =(x+y+z)^n - \sum_{F \not\in \mathcal{F}^*(V)} x^{|F_+|} y^{|F_-|} z^{|F_0|}
\end{equation}
We will show that
\begin{equation}
\label{eq:fpointedcone}
\sum_{F \not\in \mathcal{F}^*(V)} x^{|F_+|} y^{|F_-|} z^{|F_0|} = p(x+z,y+z,-z) -(-1)^d z^n
\end{equation}
Together with \eqref{eq:fstarcomplement}, this implies $\fhom^*(x,y,z) = (x+y+z)^n -(-1)^rz^n - \fhom(x+z,y+z,-z)$, which by \eqref{eq:inhom-hom} yields $f^*(z,y)=(1+y+z)^n -(-1)^r z^n - (1+z)^n f(\frac{-z}{1+z},\frac{y+z}{1+z})$. This implies \eqref{eq:fstar-f} and hence the theorem (by substituting $x$ for $z$). Thus, it remains to show \eqref{eq:fpointedcone}.

Recall that for sign vectors $F,G \in \{-1,0,+1\}^n$, we write $F\leq G$ iff $F_+ \subseteq G_+$ and $F_- \subseteq G_-$. 
We observe that, for every $G \in \{-1,0,+1\}^n$,
\begin{equation}
\label{eq:superfaces}
\sum_{F\in \{-1,0,+1\}^n, F\leq G} x^{|F_+|} y^{|F_-|} z^{|F_0|}=(x+z)^{|G_+|} (y+z)^{|G_-|}z^{|G_0|}
\end{equation}
Consider a sign vector $F\in \{-1,0,+1\}^n$. As we saw in the proof of Lemma~\ref{lem:Farkas}, we have $F\not \in \mathcal{F}^*(V)$ iff the set of vectors $\{v_i \colon i\in F_+\} \sqcup \{-v_i \colon i\in F_-\}$ is contained in an open linear halfspace. Passing to the polar dual arrangement $\arr(V)$, this means that the intersection  $C:=(\bigcap_{i \in F_+} H_i^+) \cap (\bigcap_{i \in F_-} H_i^-)$ of open hemispheres is non-empty.
Every $G\in \mathcal{F}(V)$ corresponds to a face of $\arr(V)$ of dimension $d-|G_0|$; the relative interior of this face is contained in $C$ iff $F\leq G$, and $C$ is the disjoint union of the relative interiors of these faces. 

If $F\neq 0$, then $C$ is a non-empty intersection of a non-empty collection of open hemispheres, and hence (as a non-empty, spherically convex, open region) homeomorphic to a $d$-dimensional open ball $\mathring{B}^d$ (a spherical polytope minus its boundary). By computing Euler characteristics as alternating sums of face numbers, we get
\[
\sum_{G\in \mathcal{F}(V), G\geq F} (-1)^{d-|G_0|}=\chi(B^d)-\chi(S^{d-1})=1-(1-(-1)^d)=(-1)^d
\]
hence $\sum_{G\in \mathcal{F(V)}, G\geq F}(-1)^{|G_0|}=1$ for all $F\neq 0$. 

On the other hand, if $F = 0$, then $C=S^d$, hence
\[
\sum_{G\in \mathcal{F}(V), G\geq 0} (-1)^{d-|G_0|}=\chi(S^d)=1+(-1)^d
\]
hence $\sum_{G\in \mathcal{F}(V), G\geq 0}(-1)^{|G_0|}=1+(-1)^d=\chi(S^d)$.
By combining this with \eqref{eq:superfaces}, we get
\begin{eqnarray*}
\sum_{F \not\in \mathcal{F}^*(V)} x^{|F_+|} y^{|F_-|} z^{|F_0|} &= & \sum_{F \not\in \mathcal{F}^*(V)} x^{|F_+|} y^{|F_-|} z^{|F_0|}\left( \sum_{G\in \mathcal{F}(V), G\geq F} (-1)^{|G_0|} \right)-(-1)^d z^n \\
&= & \underbrace{\sum_{G\in \mathcal{F}(V)} (x+z)^{|G_+|} (y+z)^{|G_-|} (-z)^{|G_0|}}_{=p(x+z,y+z,-z)}-(-1)^d z^n
\end{eqnarray*}
as we wanted to show.
\end{proof}

\section{Continuous Motion and the $g$-Matrix}
\label{sec:g-matrix}

\subsection{The $g$-Matrix of a Pair}
\label{sec:g-pair}

Any two configurations $V=\{v_1,\dots,v_n\}$ and $W=\{w_1,\dots,w_n\}$ of $n$ vectors in general position in $\R^r$ can be deformed into one another through a continuous family $V(t)=\{v_1(t),\dots,v_n(t)\}$ of vector configurations, where $v_i(t)$ describes a continuous path from $v_i(0)=v_i$ to $v_i(1)=w_i$ in $\R^r$. If we choose this continuous motion sufficiently generically, then there is only a finite set of events $0 < t_1 < \dots < t_N < 1$, called \emph{mutations}, during which the combinatorial type of $V(t)$ (encoded by $\mathcal{F}(V(t))$) changes, 
in a controlled way: during a mutation, a unique $r$-tuple of vectors in $V(t)$, indexed by some $R=\{i_1,\dots,i_r\} \subset [n]$, becomes linearly dependent, the orientation of the $r$-tuple (i.e., the sign of 
$\det[v_{i_1}|\dots | v_{i_r}]$) changes, and all other $r$-tuples of vectors remain linearly independent. Thus, any two vector configurations are connected by a finite sequence $V=V_0,V_1,\dots,V_N=W$ such that $V_{s-1}$ and $V_s$ differ by a mutation, $1\leq s\leq N$. 
We describe the change from $\mathcal{F}(V)$ to $\mathcal{F}(W)$ when $V$ and $W$ differ by a single mutation. Let $R\in \binom{[n]}{r}$ index the unique $r$-tuple of vectors that become linearly dependent. In terms of the polar dual arrangements, the $r$-tuple of great $(d-1)$-spheres $H_i$, $i\in R$, intersect in an antipodal pair $u,-u$ of points in $S^d$. Immediately before and immediately after the mutation, these $r$ great $(d-1)$-spheres bound an antipodal pair of simplicial $d$-faces $\sigma,-\sigma$ in $\mathcal{A}(V)$ and a corresponding pair of simplicial $d$-faces $\tau,-\tau$ in $\mathcal{A}(W)$, respectively; see Figures~\ref{fig:0k3k} and \ref{fig:1k2k} for an illustration in the case $d=2$. We have $F\in \mathcal{F}(V)\setminus \mathcal{F}(W)$ iff the face of $\arr(V)$ with signature $F$ is contained in $\sigma$ or $-\sigma$, and $F\in \mathcal{F}(W)\setminus \mathcal{F}(V)$ iff the face of $\arr(W)$  with signature $F$ is contained in $\tau$ or $-\tau$. All other faces are preserved, i.e., they belong to $\mathcal{F}(V) \cap \mathcal{F}(W)$. 

Let $Y\in \mathcal{F}(W)$ be the signature of $\tau$. We define a partition $[n] = I \sqcup J \sqcup A \sqcup B$ by 
\[I:=R \cap Y_+, \quad J:=R \cap Y_-, \quad A:=([n]\setminus R) \cap Y_+ ,\qquad B:=([n]\setminus R) \cap Y_-\]
Define $j:=|J|$ and $k:=|B|$. We call the pair $(j,k)$ the \emph{type} of the simplicial face $\tau$. The signature $X\in \mathcal{F}(V)$ of the corresponding simplicial face $\sigma$ of $\arr(V)$ satisfies $X_i = -Y_i$ for $i \in R$ and $X_i = Y_i$ for $i\in [n]\setminus R$. Thus, $\sigma$ is of type $(r-j,k)$. Analogously, $-\tau$ and $-\sigma$ are of type $(r-j,n-r-k)$ and $(j,n-r-k)$, respectively, see Figures~\ref{fig:0k3k} and \ref{fig:1k2k}. 

Let us define $f_\sigma(x,y):=\sum_{F\subseteq \sigma}  x^{|F_0|}y^{|F_-|}$, 
where we use the notation $F\subseteq \sigma$ to indicate that the sum ranges over all $F \in \mathcal{F}(V)$ corresponding to faces of $\arr(V)$ contained in $\sigma$. The polynomials $f_{-\sigma}(x,y)$,  $f_\tau(x,y)$, and $f_{-\tau}(x,y)$ are defined analogously. These four polynomials have a simple form:
\begin{eqnarray*}
& f_\sigma (x,y)=y^k\left[(x+1)^{j}(x+y)^{r-j}-x^r\right],\quad f_{-\sigma}(x,y)=y^{n-r-k}\left[(x+1)^{r-j}(x+y)^{j}-x^r\right] & \\
& f_\tau(x,y)=y^k\left[(x+1)^{r-j}(x+y)^{j}-x^r\right],\quad f_{-\tau}(x,y)=y^{n-r-k}\left[(x+1)^{j}(x+y)^{r-j}-x^r\right]
\end{eqnarray*}
We say that the mutation $V\to W$ is of \emph{Type~$(j,k)\equiv(r-j,n-r-k)$}. The reverse mutation $W\to V$ is of Type~$(r-j,k)\equiv(j,n-r-k)$.
We can summarize the discussion as follows:
\begin{lemma}
\label{lem:mutation}
Let $V\to W$ be a mutation of Type~$(j,k)\equiv(r-j,n-r-k)$ between configurations of $n$ vectors in $\R^r$. Then
\begin{equation}
\label{eq:f-g-mutation}
f_W(x,y)-f_V(x,y) = \left(y^k - y^{n-r-k}\right)\left[(x+1)^{r-j}(x+y)^{j} - (x+1)^{j}(x+y)^{r-j}\right]
\end{equation}
\end{lemma}
Note that the right-hand side of \eqref{eq:f-g-mutation} is zero if $2j=r$ or $2k=n-r$.

 \begin{figure}[ht]
\begin{center}
\includegraphics[scale=0.7]{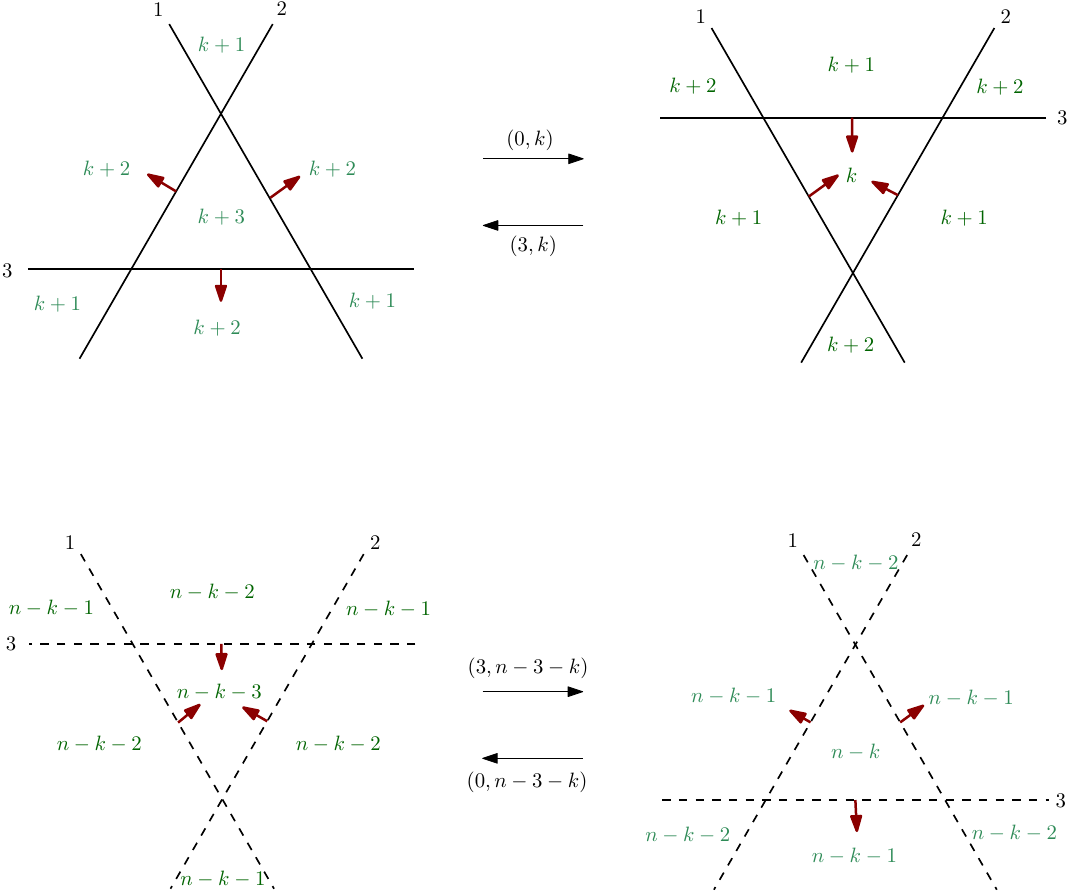}
\end{center}
\caption{A mutation of Type $(0, k) \equiv (3, n-3-k)$ (from left to right), respectively $(3, k)\equiv(0, n-3-k)$ (from right to left) in $S^2$. The upper row shows the triangular faces $\sigma$ and $\tau$ before and after the mutation, and the lower row shows the corresponding antipodal faces $-\sigma$ and $-\tau$. The little arrows indicate positive halfspaces, and the labels in full-dimensional faces indicate levels.}
\label{fig:0k3k}
\end{figure}

\begin{figure}[ht]
\begin{center}
\includegraphics[scale=0.7]{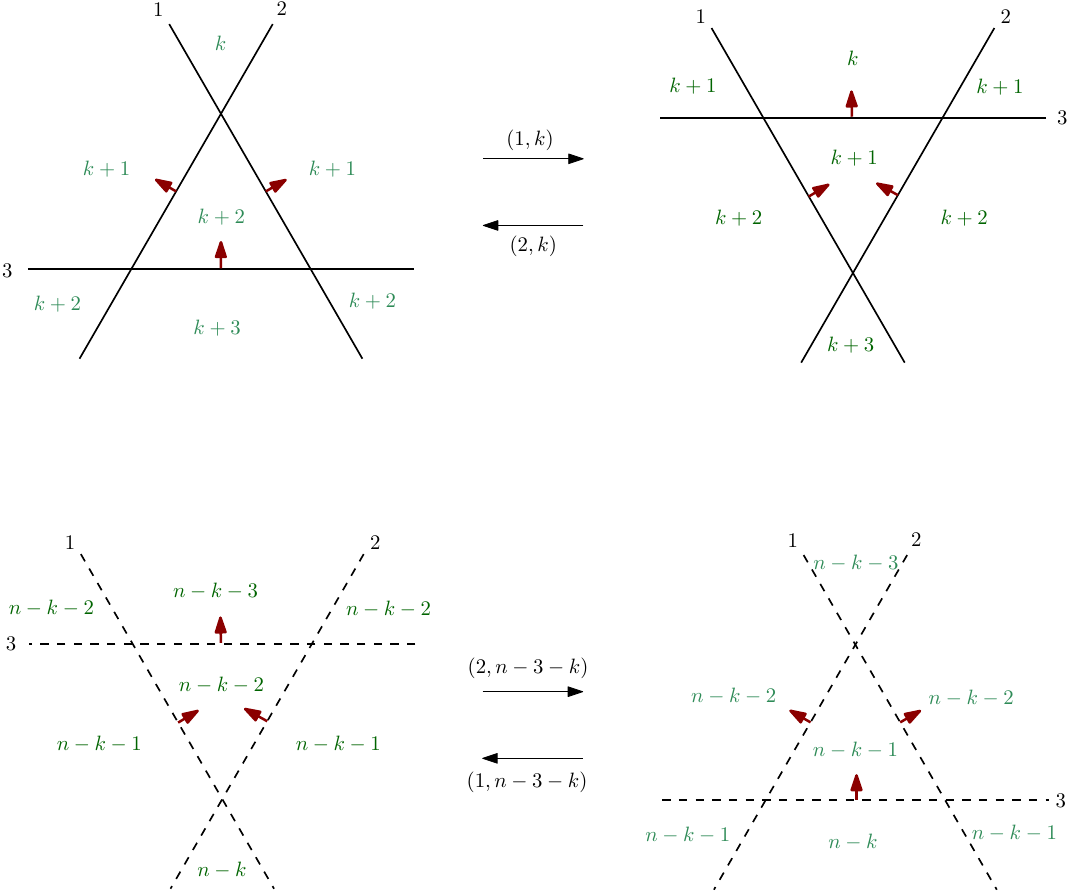}
\end{center}
\caption{A mutation of Type $(1, k)\equiv (2, n-3-k)$ (from left to right), respectively $(2, k)\equiv(1, n-3-k)$ (from right to left) in $S^2$.}
\label{fig:1k2k}
\end{figure}

We are now ready to define the $g$-matrix $g(V\rightarrow W)$ of a pair of vector configurations.  
\begin{definition}[$g$-Matrix of a pair]
\label{def:g-pair}
Let $V,W$ be configurations of $n$ vectors in $\R^r$. 

If $V\to W$ is a single mutation of Type~$(i,\ell)\equiv(r-i,n-r-\ell)$ then we define the $g$-matrix 
$g(V\rightarrow W)=[g_{j,k}(V\rightarrow W)]$, $0\leq j\leq r$ and $0\leq k \leq n-r$, as follows: 

If $2i=r$ or $2\ell=n-r$, then $g_{j,k}(V\to W)=0$ for all $j,k$. If $2i\neq r$ and $2\ell\neq n-r$, then
\[
g_{j,k}(V\to W):= 
\begin{cases}
  +1 & \textrm{if } (j,k)=(i,\ell) \textrm{ or }(j,k)=(r-i,n-r-\ell)\\
  -1 & \textrm{if } (j,k)=(r-i,\ell) \textrm{ or }(j,k)=(i,n-r-\ell)\\
  0 & \textrm{else.}  
\end{cases}
\]

More generally, if $V$ and $W$ are connected by a sequence $V=V_0, V_1,\dots, V_N=W$, where $V_{s-1}$ and $V_s$ differ by a single mutation, then we define
\[
g_{j,k}(V\to W):= \sum_{s=1}^N g_{j,k}(V_{s-1}\to V_s)
\]
\end{definition}

\begin{proof}[Proof of Thm~\ref{thm:f-g}]
All three properties follow directly from Definition~\ref{def:g-pair} and Lemma~\ref{lem:mutation}.
\end{proof}

A priori, it may seem that the definition of the $g$-matrix depends on the choice of a particular sequence of mutations transforming $V$ to $W$. However, this is not the case: 

\begin{lemma} 
\label{lem:g-f}
Let $f(x,y)=\sum_{s,t}f_{s,t}x^sy^t$ and $g(x,y)=\sum_{j,k} g_{j,k} x^j y^k$ be polynomials (with real coefficients $f_{s,t}$ and $g_{j,k}$ that are zero unless $0\leq s\leq d$ and $0\leq t\leq n$, respectively $0\leq j\leq r$ and $0\leq k\leq n-r$). Suppose that $f(x,y)$ and $g(x,y)$ satisfy the identity
\begin{equation}
\label{eq:f-g-to-be-inverted} 
f(x,y)= (1+x)^r g (\textstyle \frac{x+y}{1+x},y)=\sum_{j=0}^r \sum_{k=0}^{n-r} g_{j,k}\cdot (x+y)^j (1+x)^{r-j} y^k
\end{equation}
Then, for every fixed $t$, the numbers $g_{j,t}$, $0\leq j\leq r$, are linear combinations of the numbers $f_{s,\ell}$, $0\leq s\leq d$ and $0\leq \ell \leq t$, with coefficients given inductively by the polynomial equations 
\[
\sum_j g_{j,t} x^{r-j} = \sum_s f_{s,t} (x-1)^s - \sum_j \sum_{k<t} g_{j,k}\binom{j}{t-k}x^{r-j}
\]
\end{lemma}
\begin{proof}
The coefficient of $y^t$ in $(x+y)^j (x+1)^{r-j}y^k$ equals $\binom{j}{t-k}(x+1)^{r-j}$ (which is zero unless $0\leq k\leq t$). Thus, fixing $t$ and collecting terms in \eqref{eq:f-g-to-be-inverted} according to $y^t$, we get
\[
\sum_s f_{s,t} x^s =\sum_j \sum_{k\leq t} g_{j,k}\binom{j}{t-k}(1+x)^{r-j}
\]
Moving the terms with $k<t$ to the other side yields
\[
\sum_j g_{j,t} (1+x)^{r-j} = \sum_s f_{s,t} x^s - \sum_j \sum_{k<t} g_{j,k}\binom{j}{t-k}(1+x)^{r-j}
\]
The result follows by a change of variable from $x$ to $x-1$ (inductively, the numbers $g_{j,k}$, $k<t$, are determined by the numbers $f_{s,\ell}$, $\ell<t$.)
\end{proof}

\begin{proof}[Proof of Theorem~\ref{thm:f*-g*}] 
By Theorem~\ref{thm:f-fstar}, the $f$-polynomial and the $f^*$-polynomial of a vector configuration determine each other. By combining this with Thm.~\ref{thm:f-g}, Theorem~\ref{thm:f*-g*} follows.\end{proof}

We remark that Theorem~\ref{thm:f*-g*} can also be proved directly, by studying how $\mathcal{F}^*$ changes during mutations. By Gale duality, Thms.~\ref{thm:f-g} and \ref{thm:f*-g*} also imply the following:

\begin{corollary}
\label{cor:g-pair-dual}
Let $V,W\in \R^{r\times n}$ be vector configurations, and let $V^*,W^*\in  \R^{(n-r)\times n}$ be their Gale duals. Then $g_{j,k}(V \to W) = -g_{k,j}(V^* \to W^*)$.
\end{corollary}

As an immediate application, we show that all neighborly configurations of $n$ vectors in $\R^r$ have the same $f$-matrix (hence the same $f^*$-matrix), and analogously for coneighborly configurations.
\begin{proposition}
\label{prop:f-neighborly-coneighborly}
Let $V,W \in \R^{r\times n}$. Suppose that $V$ and $W$ are both coneighborly, or that both are neighborly. Then $g(V\to W)$ is identically zero, hence $f(V)=f(W)$, and $f^*(V)=f^*(W)$.
\end{proposition}
\begin{proof} 
If $V,W \in \R^{r\times n}$ are coneighborly configurations then $f_{s,t}(V)=f_{s,t}(W)=0$ for $t \leq \floor{\frac{n-r-1}{2}}$. By Lemma~\ref{lem:g-f}, this implies that $g_{j,k}(V\to W)=0$ for $k\leq  \floor{\frac{n-r-1}{2}}$ and $0\leq j\leq r$. The skew-symmetry of the $g$-matrix then implies  $g_{j,k}(V\to W)=0$ for all $j,k$. Thus, $f(V)=f(W)$ and $f^*(V)=f^*(W)$, by Thms.~\ref{thm:f-g} and \ref{thm:f*-g*}. The remaining statements follow by Gale duality.
\end{proof}

\subsection{Contractions and Deletions}
\label{sec:deletion-contraction}
Let $V=\{v_1,\dots,v_n\} \subset \R^r$ be a vector configuration in general position. For $1\leq i\leq n$, consider the linear hyperplane $v_i^\bot \cong \R^{r-1}$ orthogonal to $v_i$, and let $V/v_i$ denote the vector configuration obtained by projecting the remaining vectors $v_j$, $j\neq i$, orthogonally onto $v_i^\bot $. Thus, $V/v_i$ is a configuration of $(n-1)$ vectors in rank $r-1$. In terms of the polar dual arrangements, $\arr(V/i)$ is the intersection of the arrangement $\arr(V)$ with the great $(d-1)$-sphere $H_i\cong S^{d-1}$ defined by $v_i$. We call $V/v_i$ a \emph{contraction}.

Moreover, we call the configuration $V\setminus v_i$ of $n-1$ vectors in $\R^r$ obtained by removing $v_i$ a \emph{deletion}. Deletions and contractions are Gale dual to each other, i.e., $(V/v_i)^*=V^*\! \setminus \!v_i^*$.

Every generic continuous motion $V\to W$ between two vector configurations $V,W \in \R^{r\times n}$ in general position induces continuous motions $V/v_i \to W/w_i$ and $V\setminus v_i \to W\setminus w_i$, $1\leq i \leq n$.

\begin{lemma}[Contractions and Deletions]
\label{lem:g-contraction}
For $0\leq k\leq n-r$ and $0\leq j \leq r-1$,
\[
\sum_{i=1}^n g_{j,k}(V/v_i \to W/w_i)=(r-j)g_{j,k}(V\to W) + (j+1)g_{j+1,k}(V\to W)
\]
Analogously, for $0\leq k\leq n-r-1$ and $0\leq j\leq r$,
\[
\sum_{i=1}^n g_{j,k}(V \setminus  v_i \to W \setminus w_i)=(n-r-k)g_{j,k}(V\to W) + (k+1)g_{j,k+1}(V\to W)
\]
\end{lemma}
\begin{proof} We prove the formula for contractions; the result for deletions follows by Gale duality.
Consider the corresponding motion of arrangements in $S^d$. Consider a $(j,k)$-mutation in the full arrangement $\arr(V)$, involving $r=d+1$ great $(d-1)$-spheres $H_i$, $i\in R$, for some $R\in \binom{[n]}{r}$, that pass through a common point during the mutation, and that bound a small simplex before and after the mutation. Let $J\in \binom{R}{j}$ be the set of indices such that the small simplex after the mutation is contained in $H_i^-$ for $i\in J$ and in $H_i^+$ for $i\in R\setminus J$. For each $i\in R$, we see a mutation in the arrangement restricted to $H_i\cong S^{d-1}$; this mutation in $S^{d-1}$ is of type $(j-1,k)$ if $i\in J$ and of type $(j,k)$ if $i\in R\setminus J$. For $i\in [n]\setminus R$, the restriction to $H_i$ does not undergo a mutation.
\end{proof}

We say that a vector configuration $V\subset \R^r$ is \emph{$j$-neighborly} if every subset of $V$ of size $j$ is extremal. We say that $V$ is \emph{$k$-coneighborly} if $f_{s,t}=0$ for $t\leq k$, i.e., if every open linear halfspace contains at least $k+1$ vectors from $V$.
\begin{lemma}\label{gjk-nbly} Let $0\leq j\leq \frac{r-1}{2}$, $0\leq k\leq \frac{n-r-1}{2}$, and let $V,W \in \R^{r\times n}$ be vector configurations such that $V$ is $k$-coneighborly and $W$ is $j$-neighborly. Then 
\begin{equation}
\label{eq:gjk-nbly}
\textstyle g_{j,k}(V\to W)=\binom{n-k-r+j}{j}\binom{k+r-1-j}{k} - \binom{n-k-r+j-1}{j-1} \binom{k+r-j}{k} > 0
\end{equation}
\end{lemma}
The proof, which we defer to Appendix~\ref{sec:proof-gjk-nbly} is by a double induction, using the deletion and contraction formulas from Lemma~\ref{lem:g-contraction}.

\section{The Spaces $\gspace_{n,r}$ and $\gspace_{n,r}^0$ Spanned by $g$-Matrices}
\label{sec:g-space}

In this section, we prove Theorems~\ref{thm:f-space-g-space}, \ref{thm:f-space-g-space-pointed}, and \ref{thm:fstar-space-gstar-space-pointed}. By Theorems~\ref{thm:f-g} and \ref{thm:f*-g*} and Lemma~\ref{lem:g-f}, the description of the spaces $\fspace_{n,r}$, $\fspace^0_{n,r}$, and $(\fspace^*_{n,r})^0$ follows from the description of the spaces $\gspace_{n,r}$ and $\gspace_{n,r}^0$, so it remains to prove the latter. 

Recall that $\mathcal{V}_{n,r}$ is the set of all vector configurations $V\in \R^{r\times n}$ in general position, and $\mathcal{V}_{n,r}^0$ the subset of pointed configurations.

By Theorem~\ref{thm:f-g}, the $g$-matrix $g=g(V\to W)$ of any pair $V,W\in \mathcal{V}_{n,r}$ satisfies the skew-symmetries $g_{j,k}=-g_{r-j,k}=-g_{j,n-r-k}=g_{r-j,n-r-k}$ in \eqref{eq:g-skew}. Thus, in order to prove Theorem~\ref{thm:f-space-g-space}, it remains to show that $\gspace_{n,r}=\lin\{g(V\to W)\colon V,W\in \mathcal{V}_{n,r}\}$ has dimension $\floor{\frac{r+1}{2}} \floor{\frac{n-r+1}{2}}$. To see this, consider a generic continuous deformation from a coneighborly configuration $V_0$ to a neighborly configuration $V_N$, and let $V_t$, $0\leq i \leq N$,  be the intermediate vector configurations, i.e., $V_t$ and $V_{t-1}$ differ by a mutation. Thus, the $g$-matrices $g(V_0\to V_t)$ and $g(V_0\to V_{t-1})$ differ by the $g$-matrix of a mutation, i.e., their first quadrants (small $g$-matrices) differ in at most one coordinate, by $+1$ or $-1$. Moreover, $g(V_0\to V_0)$ is identically zero, and every entry of the first quadrant of $g(V_0\to V_N)$ is strictly positive by Lemma~\ref{gjk-nbly}. Thus, the proof of Theorem~\ref{thm:f-space-g-space} is completed by the following lemma:

\begin{lemma} 
\label{lem:basis1}
Let $X_0,X_1,\dots,X_N$ be vectors in $\R^m$ such that
\label{lem:basis}
\begin{enumerate}
\item $X_0=0$;
\item $X_t$ and $X_{t-1}$ differ in exactly one coordinate, by $+1$ or $-1$;
\item for every $1 \leq i \leq m$, there exists $t$ such that $X_t$ and $X_{t-1}$ differ in the $i^{\text{th}}$ coordinate
(e.g., this holds if all coordinates of $X_N$ are non-zero, by Conditions~1 and 2).
\end{enumerate}
Then there is a subset $X_{t_1}, \dots, X_{t_m}$ of vectors that form a basis of $\R^m$.
\end{lemma}

\begin{proof}
For $1 \leq i\leq m$, let $t_i$ be the smallest $t\in \{1,\dots,N\}$ such that the $i$-th coordinate of $X_{t_i}$ is non-zero; the index $t_i$ exists by Properties~1 and 3. Moreover, by Property~2, no two coordinates can become non-zero at the same time, i.e., the indices $t_i$ are pairwise distinct. Up to re-labeling the coordinates, we may assume $t_1 < t_2 < \ldots < t_m$. Then, for $1\leq i \leq m$, the vector $X_{t_i}$ is linearly independent from the vectors $X_{t_1},\dots X_{t_{i-1}}$, since all of the latter vectors have $i$-th coordinate zero. Thus, the $X_{t_i}$ form a basis.
\end{proof}

In order to prove prove Theorems~\ref{thm:f-space-g-space-pointed} and \ref{thm:fstar-space-gstar-space-pointed}, we need to prove the description \eqref{eq:g-space} of the space $\gspace_{n,r}^0=\lin\{g(V\to W)\colon V,W \in \mathcal{V}^0_{n,r}\}$. We start by observing the following:
\begin{lemma}
If $V,W \in \mathcal{V}^0_{n,r}$ then $g_{0,k}(V \to W)=0$ for all $k$.
\end{lemma}

\begin{proof}
Up to a rotation, we may assume that both $V$ and $W$ are both contained in the same open halfspace $H^+ = \{ x \in \R^r \mid \langle u ,x \rangle > 0 \}$, for some $u \in S^d$. Then $V$ can be deformed into $W$ through a continuous family $V(t)$ of vector configurations such that $V(t) \subset H^+$ and hence $0 \notin \conv (V(t))$ for all $t$. Thus, there are no mutations of types $(0,k)$ for any $k$.
\end{proof}

Thus, in order to prove \eqref{eq:g-space}, it remains to show that the space $\gspace^0_{n,r}$ has dimension $\floor{\frac{r-1}{2}}\floor{\frac{n-r+1}{2}}$. 
To this end, by Lemma~\ref{lem:basis1}, it is enough to prove the following:
\begin{lemma}
\label{lem:all-mutations}
For every $n$ and $r$, there is a sequence $V_0, V_1, \ldots V_N$ of configurations in $\mathcal{V}^0_{n,r}$ with the following properties:
\begin{enumerate}
\item $V_t$ and $V_{t+1}$ differ by a single mutation,
\item For $1 \leq j \leq \floor{\frac{r-1}{2}}$, $0 \leq k \leq \floor{\frac{n-r-1}{2}}$, some mutation $V_{t-1} \to V_t$ is of Type $(j,k)$.
\end{enumerate}
\end{lemma}

Consider a set $S = \{p_1,\dots,p_n \}\subset \R^d$ in general position, corresponding to a pointed vector configuration $V = \{v_1,\dots,v_n\} \subset \R^{d+1}$, where $v_i=(1, p_i)$. 
In order to prove Lemma~\ref{lem:all-mutations}, we will construct a continuous deformation of the point set $S$ in $\R^d$ such that the corresponding continuous deformation of $V$ in $\R^{d+1}$ contains mutations of all types $(j,k)$, for  $1 \leq j \leq \floor{\frac{d}{2}}$ and $0 \leq k \leq \floor{\frac{n-d-2}{2}}$. 

Consider a set $Q=\{q_1,\dots,q_d\}\subset \R^d$ of $d$ points in general position. We fix the labeling of the points by $i\in [d]$ and call $Q$ a \emph{labeled} point set.
Let $\Pi=\aff(Q)$ be the affine hyperplane spanned by $Q$. Every point $x\in \Pi$ can be uniquely written as an affine combination $x=\sum_{i=1}^d \alpha_i q_i$, $\alpha_i \in \R$, $\sum_i \alpha_i =1$. This defines, for every non-empty subset 
$\emptyset \neq \sigma \subset [d]$, a region  $R_\sigma(Q)=\{\sum_{i=1}^d \alpha_i q_i \colon \sum_i \alpha_i =1, \alpha_i >0 \textrm{ for } i\in \sigma, \alpha_i<0 \textrm{ for } i\in [d]\setminus \sigma\}$ such that the union of the closures 
$\bigcup_{\emptyset \neq \sigma \subset [d]} \overline{R_\sigma(Q)}$ covers all of $\Pi$.

\begin{observation}
\label{obs:point-mutation}
Let $S \subset \R^d$ be a set of $n$ points in general position, let $Q\subset S$ be a labeled subset of $d$ points, and let $p\in S\setminus Q$. Consider a continuous motion such that all points of $S\setminus \{p\}$ remain fixed and $p$ crosses the hyperplane $\Pi$ through the open region $R_{\sigma}(Q)$, from the halfspace $\Pi^-$ to the halfspace $\Pi^+$. Let $j=d-|\sigma|+1$ and $k=|\Pi^+ \cap (S\setminus (Q \sqcup \{p\}))|$. Then this corresponds to a mutation of type $(j,k)\equiv (r-j,n-r-k)$ of the corresponding pointed vector  configuration in $\R^{d+1}$.
\end{observation}

\begin{proof}[Proof of Lemma~\ref{lem:all-mutations}] 
Let $A=\{a_1,\dots,a_d\} \subset \R^d$ be a labeled point set in general position. 
For $\emptyset \neq \sigma \subseteq [d]$, choose a line $\ell_\sigma$ perpendicular to the affine hyperplane $\aff(A)$ that intersects $\aff(A)$ in a point in the open region $R_\sigma(A)$. 
Choose a small $d$-dimensional ball $B_\varepsilon$ of radius $\varepsilon>0$ centered at $a_d$. Let $q_i:= a_i$ for $1\leq i\leq d-1$. If we choose $\varepsilon$ sufficiently small, then for every $q_d\in B_\varepsilon$, the labeled set 
$Q=\{q_1,\dots,q_d\}$ has the following property: For every $\emptyset \neq \sigma \subseteq [d]$, the line $\ell_\sigma$ intersects the hyperplane $\aff(Q)$ in the interior of the region $R_\sigma(Q)$.  

Let us now set $p_i:=a_i$ for $1\leq i \leq d-1$, and choose 
$n-d$ points $p_d,\dots,p_{n-1} \in B_{\varepsilon}$ such that $P:=\{p_1,\dots,p_{n-1}\} \subset \R^d$ is in general position. These points will remain fixed throughout, and we refer to them as \emph{stationary}. 

Consider an additional point $p$ that moves continuously along one of the lines $\ell_\sigma$. During this continuous motion of $S=P\sqcup\{p\}$, we say that an \emph{interesting} mutation occurs when $p$ crosses the affine hyperplane $\Pi=\aff(Q)$ spanned by some labeled $d$-element subset $Q\subset P$ of the form $Q=\{q_1,\dots,q_d\}$ with $q_i=p_i$ for $1\leq i\leq d-1$, and $q_d\in P\setminus \{p_1,\dots, p_{d-1}\} \subset B_{\varepsilon}$. By construction and by Observation~\ref{obs:point-mutation}, every interesting mutation is of type $(j,k)\equiv (r-j,n-r-k)$, where $j=d-|\sigma|+1$ is fixed and $k=|\Pi^+ \cap P\setminus Q|$, where $\Pi^+$ is the open halfspace that $p$ enters. Thus, if we continuously move  $p$ along $\ell_\sigma$ from one side of $\conv(P)$ to the other, then every value $0\leq k \leq n-1-d$ occurs at least once. We call this the \emph{$\ell_\sigma$-stage} of the motion. We perform these $\ell_\sigma$-stages consecutively, for each of the lines $\ell_\sigma$, $\emptyset \neq \sigma \subseteq [d]$, in some arbitrary order, moving $p$ from one line to the next in between these stages in some arbitrary generic continuous motion. In the process, for every $(j,k)$ with $1\leq j \leq d$ and $0\leq k \leq n-1-d$, and interesting mutation of type $(j,k)$ will occur at least once. (We have no control over other, non-interesting mutations that occur both within the stages and in-between the different stages, but this is not necessary.)
\end{proof}

\bibliography{g-matrix}

\appendix
\section{Proof of Lemma~\ref{gjk-nbly}}
\label{sec:proof-gjk-nbly}

We recall the statement that we wish to prove: Let $0\leq j\leq \frac{r-1}{2}$, $0\leq k\leq \frac{n-r-1}{2}$, and let $V,W \in \R^{r\times n}$ be vector configurations such that $V$ is $k$-coneighborly and $W$ is $j$-neighborly. Then 
\begin{equation*}
\textstyle g_{j,k}(V\to W)=\binom{n-k-r+j}{j}\binom{k+r-1-j}{k} - \binom{n-k-r+j-1}{j-1} \binom{k+r-j}{k} > 0
\end{equation*}
We prove this by double induction. Let us first consider the case $j=0$, i.e., assume that $W$ is $0$-neighborly (i.e., pointed) and $V$ is $k$-coneighborly. We wish to show that
\begin{equation}
\label{eq:g0k-nbly}
g_{0,k}(V\to W)=\binom{k+r-1}{r-1}=\binom{k+r-1}{k}
\end{equation}
We show this by induction on $k$. Consider the base case where $V$ is $0$-coneighborly and $W$ is $0$-neighborly. Then $f_{0,0}(V)=0$ and $f_{0,0}(W)=1$, hence (using  \eqref{eq:f-g-coefficients}), $g_{0,0}(V\to W)=f_{0,0}(W)-f_{0,0}(V)=1$. 
For the induction step, assume that $k\geq 1$. We use the deletion formulas from Lemma~\ref{lem:g-contraction}. Every deletion $V\setminus v_i$ is $(k-1)$-coneighborly, and every deletion $W\setminus w_i$ remains $0$-neighborly. Thus,
\[
\underbrace{\sum_i g_{0,k-1}(V\setminus v_i \to W\setminus w_i)}_{=n\binom{k-1+r-1}{k-1}}=(n-r-k+1)\underbrace{g_{0,k-1}(V\to W)}_{=\binom{k-1+r-1}{k-1}} + k g_{0,k}(V\to W)
\]
hence $g_{0,k}(V\to W)=\frac{k+r-1}{k}\binom{k-1+r-1}{k-1}=\binom{k+r-1}{k}$, as we wanted to show.

For general $j$ and $k$, we now prove by induction on $j$ that 
\begin{equation}
\label{eq:gleqjk-nbly}
g_{\leq j, k}(V\to W) = \binom{n-k-r+j}{j} \binom{k+r-1-j}{k}
\end{equation}
which implies \eqref{eq:gjk-nbly}.
The base case $j=0$ is \eqref{eq:g0k-nbly}. Thus, assume that $j\geq 1$. We will use the contraction formulas from Lemma~\ref{lem:g-contraction}. 
\begin{claim}
\[ g_{j,k}(V \to W) = \dfrac{1}{j} \Bigl(\sum_{i=1}^n g_{\leq j-1, k}(V/v_i \to W/w_i) - r \cdot g_{\leq j-1, k}(V \to W) \Bigr) \]
\end{claim}
\begin{proof}
By Lemma~\ref{lem:g-contraction},
\begin{equation}
\label{eq:contractions}
\sum_{i=0}^n g_{l-1,k}(V/v_i\to W/w_i) = (r-l+1) \cdot g_{l-1,k}(V\to W) + l \cdot g_{l,k}(V\to W)
\end{equation}
Summing up \eqref{eq:contractions} for $1 \leq l \leq j$, we get
\[ \sum_{l=1}^j \sum_{i=1}^n g_{l-1,k}(V/v_i\to W/w_i) = \sum_{l=1}^j (r-l+1) \cdot g_{l-1,k}(V\to W) + \sum_{l=1}^j l \cdot g_{l,k}(V \to W) \]
Thus,
\[ \sum_{i=1}^n g_{\leq j-1, k}(V/v_i\to W/w_i) = r \cdot g_{\leq j-1, k}(V \to W) \underbrace{ - \sum_{l=1}^j (l-1) \cdot g_{l-1,k}(V \to W)}_{-\sum_{l=0}^{j-1} l \cdot g_{l,k}(V \to W)} + \sum_{l=1}^j l \cdot g_{l,k}(V \to W)\]
hence
\begin{equation}
\label{eq:}
\sum_{i=1}^n g_{\leq j-1, k}(V/v_i\to W/w_i) = r \cdot g_{\leq j-1, k}(V \to W) + j \cdot g_{j,k}(V \to W)
\end{equation}
as we claimed.
\end{proof}
Every contraction $V/v_i$ is $(j-1)$-neighborly and every contraction $W/w_i$ remains $k$-coneighborly, hence, inductively,
\[ g_{\leq j-1, k} (V/v_i\to W/w_i) = \binom{n-k-r+j-1}{j-1} \binom{k+r-1-j}{k} \]
Moreover, also by induction $g_{\leq j-1, k}(V \to W) = \binom{n-k-r+j-1}{j-1} \binom{k+r-j}{k}$.

Substituting both into the formula from the claim, we get
\begin{multline}
g_{\leq j,k} (V \to W) = g_{\leq j-1,k} (V \to W) + g_{j,k}(V \to W) = \dfrac{n}{j} \cdot g_{\leq j-1,k} (V/v_1 \to W/w_1) - \dfrac{r-j}{j} \cdot g_{\leq j-1,k} (V \to W) = \\
= \dfrac{n}{j} \binom{n-k-r+j-1}{j-1} \binom{k+r-1-j}{k} - \frac{r-j}{j} \binom{n-k-r+j-1}{j-1} \binom{k+r-j}{k} = \\
= \dfrac{1}{j} \binom{n-k-r+j-1}{j-1} \left( n \binom{k+r-1-j}{k} - \underbrace{(r-j)\binom{k+r-j}{k}}_{\binom{k+r-j-1}{k}(k+r-j)} \right) = \\
= \underbrace{\dfrac{1}{j} \binom{n-k-r+j-1}{j-1} \Bigl( n-k-r+j \Bigr)}_{\binom{n-k-r+j}{j}}\binom{k+r-1-j}{k}
\end{multline}
which proves \eqref{eq:gleqjk-nbly}.

It remains to show that $g_{j,k}(V \to W) > 0$. For this, we rewrite \eqref{eq:gjk-nbly} as

\begin{multline}
g_{j,k}(V \to W)=\underbrace{\binom{n-k-r+j}{j}}_{\binom{n-k-r+j-1}{j-1} \frac{n-k-r+j}{j}}\binom{k+r-1-j}{k} - \binom{n-k-r+j-1}{j-1} \underbrace{\binom{k+r-j}{k}}_{\binom{k+r-1-j}{k} \frac{k+r-j}{r-j}}  \\
= \binom{n-k-r+j-1}{j-1} \binom{k+r-1-j}{k} \underbrace{ \left( \frac{n-k-r+j}{j} - \frac{k+r-j}{r-j} \right)}_{>0}
\end{multline}

To see that the expression in parentheses is positive, we notice $(n-k-r+j)(r-j) - j (k+r-j) = (n-k-r)(r-j)-jk$ and $n-k-r > k$, $r-j > j$ for the specified ranges of $j$ and $k$.
This completes the proof of Lemma~\ref{gjk-nbly}.\qed

\section{A Proof of the Dehn--Sommerville Relations for Levels}
\label{app_proof-DS}
The proof of the Dehn--Sommerville relations for levels in simple arrangements uses ideas closely related to the proof of Theorem~\ref{thm:f-fstar}. For completenes, we include the argument here.

\begin{proof}[Proof of Thm.~\ref{thm:DS-levels}]
Let $\fhom(x,y,z):=\sum_{F\in \mathcal{F}(V)} x^{|F_+|}y^{|F_-|}z^{|F_0|}=x^n f_V(\frac{z}{x},\frac{y}{x})$ be the homogeneous version of the $f$-polynomial defined in the proof of Theorem~\ref{thm:f-fstar}. 

Let $F,G \in \mathcal{F}(V)$ be the signatures of faces in $\arr(V)$. We observe that the face with signature $F$ is contained in the face with signature $G$ iff $F\leq G$.

For every $G \in \mathcal{F}(V)$, the corresponding face combinatorially a polytope, hence its Euler characteristic equals 
$1=\sum_{F\in \mathcal{F}(V), F\leq G} (-1)^{d-|F_0|}$.
Moreover, since the arrangement is simple, we get that for every $F\in \mathcal{F}(V)$,
$
\sum_{G\in \mathcal{F}(V), G\geq F} x^{|G_+|} y^{|G_-|} z^{|G_0|}=x^{|F_+|} y^{|F_-|}(x+y+z)^{|F_0|}
$.
Combining these two observations, we get
\begin{eqnarray*}
\fhom(x,y,z) & = & \sum_{G\in \mathcal{F}(V)}x^{|G_+|} y^{|G_-|} z^{|G_0|} = \sum_{G\in \mathcal{F}(V)}\left(\sum_{F\in \mathcal{F}(V), F\leq G} (-1)^{d-|F_0|}\right)x^{|G_+|} y^{|G_-|}z^{|G_0|}\\
& = & (-1)^d \sum_{F\in \mathcal{F}(V)} (-1)^{|F_0|} \left( \sum_{G\in \mathcal{F}(V), G\geq F} x^{|G_+|} y^{|G_-|}z^{|G_0|}
\right)\\
& = & (-1)^d \sum_{F\in \mathcal{F}(V)} x^{|F_+|} y^{|F_-|}(x+y+z)^{|F_0|}(-1)^{|F_0|} = (-1)^d p(x,y,-(x+y+z))
\end{eqnarray*}
Thus, by \eqref{eq:inhom-hom}, 
$
f_V(x,y)=p(1,y,x)=(-1)^d p(1,y,-(x+y+1))=(-1)^d f_V(-(x+y+1),y).
$
\end{proof}

\end{document}